\documentclass[12pt,reqno]{amsart}

\usepackage{amssymb,amsmath,graphicx,amsfonts,euscript}
\usepackage{color}

\newtheorem{thm}{Theorem}[section]

\newtheorem{prop}[thm]{Proposition}
\newtheorem{define}[thm]{Definition}
\newtheorem{rem}[thm]{Remark}

\newtheorem{lemma}[thm]{Lemma}

\newcommand\R{\mathbb {R}}
\newcommand\Z{\mathbb {Z}}

\numberwithin{equation}{section}
\begin{document}
\title[Global solutions to a tropical climate model]{Global small solutions to a tropical climate model without thermal diffusion}

\author[Renhui Wan]{ Renhui Wan$^{\ast}$}

\address{$^\ast$ Department of Mathematics,
Zhejiang University,
Hanzhou 310027, China}

\email{rhwanmath@zju.edu.cn, rhwanmath@163.com, 21235002@zju.edu.cn}

\vskip .2in
\begin{abstract}
We obtain the global well-posedness of classical solutions to a tropical climate model derived by Feireisl-Majda-Pauluis in \cite{FMP} with only the dissipation of the first
baroclinic model of the velocity ($-\eta \Delta v$) under small initial data. The main difficulty is the absence of thermal diffusion.
To overcome it,  we exploit the structure of the equations coming from the coupled terms, dissipation term and damp term.  Then  we find the hidden thermal diffusion.  In addition, based on the Littlewood-Palay theory, we establish a generalized commutator estimate, which  may be applied to other partial differential equations.

\end{abstract}

\maketitle

\vskip .2in
\section{Introduction}
\label{Introduction}
The purpose of this article is to study the cauchy problem for a tropical model without thermal diffusion:
\begin{equation} \label{TCM}
\left\{
\begin{aligned}
& \partial_t u + u\cdot\nabla u +\alpha u +\nabla p = -{\rm div}(v\otimes v),  \\
& \partial_t v + u\cdot\nabla v + v\cdot\nabla u+\alpha v-\eta\Delta v =\nabla \theta,\\
& \partial_t \theta+u\cdot\nabla \theta={\rm div}v,\\
& {\rm div} u=0,\\
& (u(0,x),v(0,x),\theta(0,x))=(u_0(x),v_0(x),\theta_0(x)),
\end{aligned}
\right.
\end{equation}
here $(t,x)\in \mathbb{R}^{+}\times \mathbb{R}^2$,  $u=(u^1,u^2),$ $v=(v^1,v^2)$ stand for the barotropic mode and the first baroclinic mode
of the vector velocity, respectively, $p,\theta$ represent the scalar pressure, scalar temperature, respectively, $\alpha$ and $\eta$
 are the nonnegative parameters.
\vskip .1in
By performing a Galerkin truncation to the hydrostatic Boussinesq equations, Feireisl-Majda-Pauluis in  \cite{FMP} derived a version of   (\ref{TCM}) without any Laplacian terms,
of which the first  baroclinic mode  had been originally used in some studies of tropical atmosphere in \cite{Gill} and \cite{Matsuno}.
For more details on the first  baroclinic mode, we refer to the section 1 and section 2 in \cite{FMP} and references therein.

\vskip .1in
Recently, for the version of (\ref{TCM}) with $-\Delta u$, $\alpha=0$ and $\eta=1$, Li-Titi in \cite{LT} obtained the global well-posedness without any small assumptions
of initial data. The difficulty of their work  is that  energy method can not be applied to  get the gradient estimate of $(u,v,\theta)$ directly due to the absence of thermal diffusion. However, by introducing  a unknown $w$,
$$w:\stackrel{def}{=} v-\nabla(-\Delta)^{-1}\theta,$$
they overcome this difficulty and  improved the regularity of $u$, and then obtained the gradient estimate of $(u,v,\theta)$. It is clear  that the Laplacian term,
$-\Delta u$, plays a key role, while  the method  does not appear to be able to extend to the case without $-\Delta u$ even if the initial data is small.

\vskip .1in
However, for (\ref{TCM}) with small data, we can get the global well-posedness which is the main result of this paper.  The details can be given as follows:
\begin{thm}\label{main}
Let $\alpha>0$ and $\eta>0$. Consider (\ref{TCM}) with initial data $(u_0,v_0, \theta_0)\in H^s(\R^2)$, $s>2,$ and ${\rm div}u_0=0$. There exists a small
constant $\epsilon=\epsilon(\alpha,\eta)>0$ such that if
\begin{equation}\label{small}
\|u_0\|_{H^s(\R^2)}+\|v_0\|_{H^s(\R^2)}+\|\theta_0\|_{H^s(\R^2)}<\epsilon,
\end{equation}
then (\ref{TCM}) admits a unique global  solution $(u,v,\theta)$ satisfying
$$(u,v,\theta)\in C([0,\infty);H^s(\R^2)),\ \nabla v\in L^2([0,\infty);H^{s}(\R^2)).$$
\end{thm}
\begin{rem}\label{r1}
($i$) If we neglect the coupled terms $\nabla \theta$ and ${\rm div}v$, the key part of (\ref{TCM}) is the first two equations, which  are very similar to the 2D MHD equations. To the best of our knowledge for that without velocity dissipation and small data, the global regularity result  is empty ( see, e.g., \cite{Wu1},\cite{Wu2},\cite{F5} and references therein), from which, it seems very difficult to drop the condition (\ref{small}).

($ii$) Motivated by the recent works on the local well-posedness  for the non-resistive MHD equations (i.e., only with $-\Delta u$) with low regular initial data (see, e.g., \cite{Chemin},\cite{Fefferman}
and \cite{Wan}),  we  expect that similar result holds for  the version of (\ref{TCM}) with $-\Delta u$,
 $\alpha=0$ and $\eta=0$.
\end{rem}

\vskip .2in
Now, let us explain the difficulty and our idea. By the standard energy method, we can show  that
\begin{equation}\label{1}
\begin{aligned}
\frac{1}{2}\frac{d}{dt}&\|(u,v,\theta)\|_{H^s(\R^2)}^2+\alpha\|(u,v)\|_{H^s(\R^2)}^2+\frac{\eta}{2}\|\nabla v\|_{H^s(\R^2)}^2\\
&\le C(\eta)\left\{\|(\nabla u,\nabla v,\nabla \theta)\|_{L^\infty(\R^2)}+\|v\|_{L^\infty(\R^2)}^2\right\}\|(u,v,\theta)\|_{\dot{H}^s(\R^2)}^2,
\end{aligned}
\end{equation}
from which, we can  see  that (\ref{1}) does not be closed under small initial data unless  some norm of $\theta$ such as $\|\theta\|_{H^s(\R^2)}$ occurs on the left hand side of (\ref{1}).
\vskip .1in
Our proof is exploiting the structure of (\ref{TCM}). To make our idea clear, we give the details for the key part of linearized (\ref{TCM}):
\begin{equation} \label{linear}
\left\{
\begin{aligned}
& \partial_t v +\alpha v-\eta\Delta v-\nabla \theta =0,\\
& \partial_t \theta-{\rm div}v=0.\\
\end{aligned}
\right.
\end{equation}
Applying the operator $\Lambda^{-1}{\rm div}$ and $\eta\Lambda$ to the first and second equation of (\ref{linear}), respectively,
 then adding the resulting equations, denote
 $$\mathcal{R}:\stackrel{def}{=} \Lambda^{-1}{\rm div},\ \ \Omega:\stackrel{def}{=} \mathcal{R}v+\eta \Lambda \theta,$$
it is easy to  deduce
\begin{equation}\label{2}
\partial_t \Omega+\frac{1}{\eta}\Omega=(\frac{1}{\eta}-\alpha) \mathcal{R}v.
\end{equation}
Multiplying the first equation of (\ref{linear}) by a large enough constant $M$, adding the resulting equation to (\ref{2}) and combining with the $L^2$ bound of  Riesz transform, we can find the hidden thermal diffusion and then overcome this difficulty.
\vskip .1in
Let us complete this section by describing the notations we shall use in this paper.\\
{\bf Notations} For $A$, $B$ two operator, we denote $[A,B]=AB-BA$, the commutator between $A$ and $B$. The uniform constant  $C$, which may be different on different lines,  is independent of the parameters such as $\alpha$ and $\eta$ in (\ref{TCM}), while the constant $C(\cdot)$ means a constant depends on the element(s) in bracket. In some places of this paper,  we may use $L^p$, $\dot{H}^s$ ($H^s$) and $\dot{B}_{p,r}^s$ ($B_{p,r}^s$) to stand for  $L^p(\R^d)$, $\dot{H}^s(\R^d)$ ($H^s(\R^d)$) and $\dot{B}_{p,r}^s(\R^d)$ ($B_{p,r}^s(\R^d)$), respectively.   We shall denote by $(a|b)$ the $L^2$ inner product
of $a$ and $b$, and $(a|b)_{\dot{H}^s}$ stands for the standard $\dot{H}^s$ inner product of $a$ and $b$, more precisely,
$(a|b)_{\dot{H}^s}=(\Lambda^s a|\Lambda^s b).$

\vskip .3in
\section{Preliminaries}
\label{sec:Preliminaries}
In this section, we give some necessary definitions and  propositions.
\vskip .1in
The fractional Laplacian operator $\Lambda^\alpha=(-\Delta)^\frac{\alpha}{2}$ $(\alpha\ge0)$ is defined through the Fourier transform, namely,
$$\widehat{\Lambda^\alpha f}(\xi)=|\xi|^\alpha \widehat{f}(\xi),$$
where the Fourier transform is given by
$$\widehat{f}(\xi)=\int_{\R^d}e^{-ix\cdot\xi}f(x)dx.$$

Let $\mathfrak{B}=\{\xi\in\mathbb{R}^d,\ |\xi|\le\frac{4}{3}\}$ and $\mathfrak{C}=\{\xi\in\mathbb{R}^d,\ \frac{3}{4}\le|\xi|\le\frac{8}{3}\}$. Choose two nonnegative smooth radial function $\chi,\ \varphi$ supported, respectively, in $\mathfrak{B}$ and $\mathfrak{C}$ such that
$$\chi(\xi)+\sum_{j\ge0}\varphi(2^{-j}\xi)=1,\ \ \xi\in\mathbb{R}^d,$$
$$\sum_{j\in\mathbb{Z}}\varphi(2^{-j}\xi)=1,\ \ \xi\in\mathbb{R}^d\setminus\{0\}.$$
We denote $\varphi_{j}=\varphi(2^{-j}\xi),$ $h=\mathfrak{F}^{-1}\varphi$ and $\tilde{h}=\mathfrak{F}^{-1}\chi,$ where $\mathfrak{F}^{-1}$ stands for the inverse Fourier transform. Then the dyadic blocks
$\Delta_{j}$ and $S_{j}$ can be defined as follows
$$\Delta_{j}f=\varphi(2^{-j}D)f=2^{jd}\int_{\mathbb{R}^d}h(2^jy)f(x-y)dy,$$
$$S_{j}f=\sum_{k\le j-1}\Delta_{k}f=\chi(2^{-j}D)f=2^{jd}\int_{\mathbb{R}^d}\tilde{h}(2^jy)f(x-y)dy.$$
Formally, $\Delta_{j}=S_{j}-S_{j-1}$ is a frequency projection to annulus $\{\xi:\ C_{1}2^j\le|\xi|\le C_{2}2^j\}$, and $S_{j}$ is a frequency projection to the ball $\{\xi:\ |\xi|\le C2^j\}$. One  easily verifies that with our choice of $\varphi$
$$\Delta_{j}\Delta_{k}f=0\ {\rm if} \ |j-k|\ge2\ \ {\rm and}\ \  \Delta_{j}(S_{k-1}f\Delta_{k}f)=0\  {\rm if}\  |j-k|\ge5.$$
Let us recall the definition of the  Besov space.

\begin{define}\label{HB}
 Let $s\in \mathbb{R}$, $(p,q)\in[1,\infty]^2,$ the homogeneous Besov space $\dot{B}_{p,q}^s(\R^d)$ is defined by
$$\dot{B}_{p,q}^{s}(\R^d)=\{f\in \mathfrak{S}'(\R^d);\ \|f\|_{\dot{B}_{p,q}^{s}(\R^d)}<\infty\},$$
where
\begin{equation*}
\|f\|_{\dot{B}_{p,q}^s(\R^d)}=\left\{\begin{aligned}
&\displaystyle (\sum_{j\in \mathbb{Z}}2^{sqj}\|\Delta_{j}f\|_{L^p(\R^d)}^{q})^\frac{1}{q},\ \ \ \ {\rm for} \ \ 1\le q<\infty,\\
&\displaystyle \sup_{j\in\mathbb{Z}}2^{sj}\|\Delta_{j}f\|_{L^p(\R^d)},\ \ \ \ \ \ \ \ {\rm for}\ \ q=\infty,\\
\end{aligned}
\right.
\end{equation*}
and $\mathfrak{S}'(\R^d)$ denotes the dual space of $\mathfrak{S}(\R^d)=\{f\in\mathcal{S}(\mathbb{R}^d);\ \partial^{\alpha}\hat{f}(0)=0;\ \forall\ \alpha\in \ \mathbb{N}^d $\ {\rm multi-index}\} and can be identified by the quotient space of $\mathcal{S'}/\mathcal{P}$ with the polynomials space $\mathcal{P}$.
\end{define}
\begin{define}\label{INB}
 Let $s>0,$ and $(p,q)\in [1,\infty]^2$, the inhomogeneous Besov space $B_{p,q}^s(\R^d)$ is defined by
$${B}_{p,q}^{s}(\R^d)=\{f\in \mathcal{S'}(\mathbb {R}^d);\ \|f\|_{{B}_{p,q}^{s}(\R^d)}<\infty\},$$
where
$$\|f\|_{B_{p,q}^s(\R^d)}=\|f\|_{L^p(\R^d)}+\|f\|_{\dot{B}_{p,q}^s(\R^d)}.$$
\end{define}
For the special case $p=q=2$, we have
$$\|f\|_{\dot{H}^s(\R^d)} \thickapprox \|f\|_{\dot{B}_{2,2}^s(\R^d)},$$
where $a\thickapprox b$ means $C^{-1}b\le a\le C b$ for some positive  constant $C$,
and the $\dot{H}^s(\R^d)$  and $H^s(\R^d)$ ($s>0$) norm of $f$ can be also defined as follows:
$$\|f\|_{\dot{H}^s(\R^d)}:\stackrel{def}{=} \|\Lambda^s f\|_{L^2(\R^d)}$$
and
$$\|f\|_{H^s(\R^d)}:\stackrel{def}{=} \| f\|_{L^2(\R^d)}+\|\Lambda^s f\|_{L^2(\R^d)}.$$
\begin{lemma}($i$)\cite{KP}
Let $s>0$, $1\le p,r\le \infty,$ then
\begin{equation}\label{kp}
\|fg\|_{\dot{B}_{p,r}^{s}(\R^d)}\le C\left\{\|f\|_{L^{p_{1}}(\R^d)}\|g\|_{\dot{B}_{p_{2},r}^{s}(\R^d)}+\|g\|_{L^{r_{1}}(\R^d)}\|g\|_{\dot{B}_{r_{2},r}^{s}(\R^d)}\right\},
\end{equation}
where $1\le p_{1},r_{1}\le \infty$ such that $\frac{1}{p}=\frac{1}{p_{1}}+\frac{1}{p_{2}}=\frac{1}{r_{1}}+\frac{1}{r_{2}}$.\\
($ii$)\cite{KPV} Let $s>0$, and $1<p<\infty$,  then
\begin{equation}\label{kpv}\|[\Lambda^s,f]g\|_{L^p(\R^d)}\le C\left\{\|\nabla f\|_{L^{p_1}(\R^d)}\|\Lambda^{s-1}g\|_{L^{p_2}(\R^d)}+\|\Lambda^{s}f\|_{L^{p_3}(\R^d)}\|g\|_{L^{p_4}(\R^d)}\right\}
\end{equation}
where $1<p_2,p_3<\infty$ such that $\frac{1}{p}=\frac{1}{p_1}+\frac{1}{p_2}=\frac{1}{p_3}+\frac{1}{p_4}$.
\end{lemma}
 The following proposition and lemma
provide Bernstein type inequalities for fractional derivatives and standard commutator estimate.
\begin{prop}
Let $\gamma\ge0$. Let $1\le p\le q\le \infty$.
\begin{enumerate}
\item[1)] If $f$ satisfies
$$
\mbox{supp}\, \widehat{f} \subset \{\xi\in \mathbb{R}^d: \,\, |\xi|
\le \mathcal{K} 2^j \},
$$
for some integer $j$ and a constant $\mathcal{K}>0$, then
$$
\|(-\Delta)^\gamma f\|_{L^q(\mathbb{R}^d)} \le C_1(\gamma,p,q)\, 2^{2\gamma j +
j d(\frac{1}{p}-\frac{1}{q})} \|f\|_{L^p(\mathbb{R}^d)}.
$$
\item[2)] If $f$ satisfies
\begin{equation*}\label{spp}
\mbox{supp}\, \widehat{f} \subset \{\xi\in \mathbb{R}^d: \,\, \mathcal{K}_12^j
\le |\xi| \le \mathcal{K}_2 2^j \}
\end{equation*}
for some integer $j$ and constants $0<\mathcal{K}_1\le \mathcal{K}_2$, then
$$
C_1(\gamma,p,q)\, 2^{2\gamma j} \|f\|_{L^q(\mathbb{R}^d)} \le \|(-\Delta)^\gamma
f\|_{L^q(\mathbb{R}^d)} \le C_2(\gamma,p,q)\, 2^{2\gamma j +
j d(\frac{1}{p}-\frac{1}{q})} \|f\|_{L^p(\mathbb{R}^d)}.
$$
\end{enumerate}
\end{prop}

\begin{lemma} \cite{BCD}
Let $\theta$ be a $C^1$ function on $\R^d$ such that $(1+|\cdot|)\widehat{\theta}\in L^1(\R^d).$ There exists a constant $C$ such that for any Lipschitz function
$a$ with gradient in $L^p(\R^d)$ and any function $b$ in $L^q(\R^d)$, we have for any positive $\lambda$,
\begin{equation}\label{SCE}
\|[\theta(\lambda^{-1}D),a]b\|_{L^r(\R^d)}\le C\lambda^{-1}\|\nabla a\|_{L^p(\R^d)}\|b\|_{L^q(\R^d)}, \ \ {\rm with}\ \frac{1}{p}+\frac{1}{q}=\frac{1}{r}.
\end{equation}
\end{lemma}

For more details about Besov space and Sobolev space such as some useful embedding inequalities, we refer to
\cite{BCD}, \cite{Grafakos} and \cite{Stein}.
\vskip .1in

The rest of this section is devoted to the proof of a generalized commutator estimate in Besov space. Firstly, we need a lemma.
\begin{lemma}\label{l1}
Let $1\le p,p_1,p_2\le \infty$ satisfying $1+\frac{1}{p}=\frac{1}{p_1}+\frac{1}{p_2}.$ If $xh\in L^{p_1}(\R^d)$, $\nabla f\in L^\infty (\R^d)$ and
$g\in L^{p_2}(\R^d)$, then
\begin{equation}\label{lWu}
\|h\star(fg)-f(h\star g)\|_{L^{p}(\R^d)}\le C\|xh\|_{L^{p_1}(\R^d)}\|\nabla f\|_{L^\infty(\R^d)}\|g\|_{L^{p_2}(\R^d)},
\end{equation}
where $C$ is a constant independent of $f,g,h$.
\end{lemma}
\begin{proof}[Proof of Lemma \ref{l1}]
 (\ref{lWu}) can be proved by using the idea of Lemma 2.1 in \cite{WuJMFM}, so we omit the details.
\end{proof}

\begin{prop}\label{pp}
Let $s\ge0$, $\sigma>-1$ and $1\le p,r\le \infty$, then
\begin{equation}\label{CE}
\|[\Lambda^s,f\cdot\nabla]g\|_{\dot{B}_{p,r}^\sigma(\R^d)}\le C\left\{\|\nabla f\|_{L^\infty(\R^d)}\|g\|_{\dot{B}_{p,r}^{\sigma+s}(\R^d)}
+\|\nabla g\|_{L^\infty(\R^d)}\|f\|_{\dot{B}_{p,r}^{\sigma+s}(\R^d)}\right\},
\end{equation}
where ${\rm div}f=0$ and the constant $C$ is independent of $f$ and $g$.
\end{prop}
For the proof, we shall use homogeneous Bony's decomposition:
$$uv=\sum_{j\in \mathbb{Z}}S_{j-1}u\Delta_{j}v+\sum_{j\in \mathbb{Z}}\Delta_{j}uS_{j-1}v+\sum_{j\in\mathbb{Z}}\Delta_{j}u\tilde{\Delta}_{j}v,$$
where $\tilde{\Delta}_{j}=\Delta_{j-1}+\Delta_{j}+\Delta_{j+1},$ which is   applied to split the commutator $\Theta=[\Delta_{j}, u]v$ as follows:
\begin{equation*}
\begin{aligned}
\Theta=&\sum_{|k-j|\le 4}[\Delta_{j},S_{k-1}u]\Delta_{k}v+\sum_{|k-j|\le 4}\Delta_{j}(\Delta_{k}u\ S_{k-1}v)\\
&+\sum_{k\ge j-2}\Delta_{k}u\ \Delta_{j}S_{k+2}v+\sum_{k\ge j-3}\Delta_{j}(\Delta_{k}u\ \tilde{\Delta}_{k}v).
\end{aligned}
\end{equation*}
If we replace $\Lambda^s$ by Riesz operator $\Lambda^{-1}\partial_1$ or the operator $\Lambda^{-\alpha}\partial_1$ ($0<\alpha<1$) in (\ref{CE}),   \cite{Hmidi1}-\cite{JMWZ}  established some similar estimates, which play the essential role in the proof of the global well-posedness for 2D Boussinesq equations.
\begin{proof}[Proof of Proposition \ref{pp}]
It suffices to prove the case $1\le r<\infty$, the case $r=\infty$ can be bounded similarly.  In this proof, $(c_{j})_{j\in\mathbb{Z}}$  is a generic element of $l^r(\mathbb{Z})$ so that $\sum_{j\in\Z}c_j^r\le1$. {\bf 1} stands for the characteristic function.  We split the left hand side of (\ref{CE}) into two terms.
\begin{equation}\label{2.1}
\begin{aligned}
\|[\Lambda^s,&f\cdot\nabla]g\|_{\dot{B}_{p,r}^\sigma}\le C\left(\sum_{j\in \Z} 2^{j\sigma r}\|\Delta_j(f\cdot\nabla \Lambda^s g)-f\cdot\nabla \Delta_j \Lambda^s g\|_{L^p}^r\right)^\frac{1}{r}\\
 &+C\left(\sum_{j\in \Z} 2^{j\sigma r}\|\Delta_j \Lambda^s(f\cdot\nabla g)-f\cdot\nabla \Delta_j \Lambda^s g\|_{L^p}^r\right)^\frac{1}{r}\\
\le& C\left(\sum_{j\in\Z}2^{j\sigma r}\|[\Delta_j, f\cdot\nabla]\Lambda^sg\|_{L^p}^r\right)^\frac{1}{r}+C\left(\sum_{j\in\Z}2^{j\sigma r}\|[\Delta_j\Lambda^s, f\cdot\nabla]\Lambda^sg\|_{L^p}^r\right)^\frac{1}{r}\\
=& K_1+K_2.
\end{aligned}
\end{equation}
{\bf The Estimate of $K_1$.} Using the homogeneous Bony's decomposition,
\begin{equation*}
\begin{aligned}
\|[\Delta_j,f\cdot\nabla]\Lambda^sg\|_{L^p}\le& \sum_{|k-j|\le 4}\|[\Delta_j,S_{k-1}f\cdot\nabla]\Delta_k\Lambda^sg\|_{L^p}
+\sum_{|k-j|\le 4}\|\Delta_j(\Delta_k f\cdot\nabla S_{k-1}\Lambda^s g)\|_{L^p}\\
&+\sum_{k\ge j-2}\|\Delta_k f\cdot\nabla \Delta_j S_{k+2}\Lambda^s g\|_{L^p}
+\sum_{k\ge j-3}\|\Delta_j(\Delta_k f\cdot\nabla \tilde{\Delta}_k \Lambda^s g)\|_{L^p}\\
=&K_{11}+K_{12}+K_{13}+K_{14}.
\end{aligned}
\end{equation*}
Thanks to (\ref{SCE}) and Bernstein's inequality,
\begin{equation*}
\begin{aligned}
K_{11}\le& C\sum_{|k-j|\le4}\|\nabla S_{k-1}f\|_{L^\infty}\|\Delta_k \Lambda^s g\|_{L^p}\\
\le& C2^{-j\sigma}\|\nabla f\|_{L^\infty}\sum_{|k-j|\le4} 2^{(j-k)\sigma}2^{k\sigma}\|\Delta_k \Lambda^s g\|_{L^p}\\
\le& C2^{-j\sigma}\|\nabla f\|_{L^\infty}\|\Lambda^s g\|_{\dot{B}_{p,r}^\sigma}\sum_{|k-j|\le4} 2^{(j-k)\sigma}c_k\\
\le& Cc_j 2^{-j\sigma}\|\nabla f\|_{L^\infty}\| g\|_{\dot{B}_{p,r}^{\sigma+s}}.
\end{aligned}
\end{equation*}
By H\"{o}lder's inequality and Bernstein's inequality, we get for $s\ge 0$,
\begin{equation*}
\begin{aligned}
K_{12}\le& C\sum_{|k-j|\le4}2^{ks}\|\nabla S_{k-1}g\|_{L^\infty}\|\Delta_k f\|_{L^p}\\
\le& C\|\nabla g\|_{L^\infty}\sum_{|k-j|\le4}2^{ks}\|\Delta_k f\|_{L^p}\\
\le& C2^{-j\sigma}\|\nabla g\|_{L^\infty}\sum_{|k-j|\le4}2^{(j-k)\sigma}2^{k(\sigma+s)}\|\Delta_k f\|_{L^p}\\
\le& C2^{-j\sigma}\|\nabla g\|_{L^\infty}\| f\|_{\dot{B}_{p,r}^{\sigma+s}}\sum_{|k-j|\le4}2^{(j-k)\sigma}c_k\\
=& Cc_j2^{-j\sigma}\|\nabla g\|_{L^\infty}\| f\|_{\dot{B}_{p,r}^{\sigma+s}}.
\end{aligned}
\end{equation*}
For the term $K_{13}$,
\begin{equation*}
\begin{aligned}
K_{13}\le& C2^{j(1+s)}\|\Delta_j g\|_{L^p}\sum_{k\ge j-2}\|\Delta_k f \|_{L^\infty}\\
\le& C2^{j(1+s)}\|\Delta_j g\|_{L^p}\sum_{k\ge j-2}2^{-k}\|\nabla \Delta_k f \|_{L^\infty}\\
\le& C2^{js}\|\Delta_j g\|_{L^p}\|\nabla f\|_{L^\infty}\\
=& Cc_j 2^{-j\sigma}\|\nabla f\|_{L^\infty}\|g\|_{\dot{B}_{p,r}^{s+\sigma}}.
\end{aligned}
\end{equation*}
Using ${\rm  div} f=0$, Bernstein's inequality and H\"{o}lder's inequality,
\begin{equation*}
\begin{aligned}
K_{14}\le& 2^j \sum_{k\ge j-3}\|\Delta_k f\|_{L^\infty}\|\Lambda^s\tilde{\Delta}_k g \|_{L^p}\\
\le& C2^{-j\sigma}\sum_{k\ge j-3}2^{j(\sigma+1)-k}\|\nabla\Delta_k f\|_{L^\infty}\|\Lambda^s\tilde{\Delta}_k g \|_{L^p}\\
\le& C2^{-j\sigma}\|\nabla f\|_{L^\infty}\sum_{k\ge j-3}2^{(j-k)(\sigma+1)}2^{k\sigma}\|\Lambda^s\tilde{\Delta}_k g \|_{L^p}\\
\le& C2^{-j\sigma}\|\nabla f\|_{L^\infty}\| g\|_{\dot{B}_{p,r}^{\sigma+s}}\sum_{k\ge j-3}2^{(j-k)(\sigma+1)}c_k\\
=& Cc_j2^{-j\sigma}\|\nabla f\|_{L^\infty}\| g\|_{\dot{B}_{p,r}^{\sigma+s}},
\end{aligned}
\end{equation*}
where we have used Young's inequality for series for the last equality, namely, $\forall \sigma>-1$,
$$\sum_{k\in\mathbb{Z}}\sum_{k\ge j-3}2^{(j-k)(\sigma+1)}c_k\le \|2^{-(\sigma+1)k}{\bf 1}_{k\ge -3}\|_{l^1(\mathbb{Z})}\|c_k\|_{l^1(\mathbb{Z})}\le C.$$
Thus,
$$
K_1\le \sum_{1\le i\le 4}(\sum_{j\in\Z}2^{j\sigma r}K_{1i}^r)^\frac{1}{r}\le C\left\{\|\nabla f\|_{L^\infty}\| g\|_{\dot{B}_{p,r}^{\sigma+s}}+\|\nabla g\|_{L^\infty}\| f\|_{\dot{B}_{p,r}^{\sigma+s}}\right\}.
$$
{\bf The Estimate of $K_2$.} Using the homogeneous Bony's decomposition again,
\begin{equation*}
\begin{aligned}
\|[\Delta_j\Lambda^s,f\cdot\nabla]g\|_{L^p}\le& \sum_{|k-j|\le4}\|[\Delta_j\Lambda^s,S_{k-1}f\cdot\nabla]\Delta_kg\|_{L^p}
+\sum_{|k-j|\le4}\|\Delta_j\Lambda^s (\Delta_k f\cdot\nabla S_{k-1}g)\|_{L^p}\\
&+\sum_{k\ge j-2}\|\Delta_k f\cdot\nabla \Delta_j S_{k+2}\Lambda^s g\|_{L^p}
+\sum_{k\ge j-3}\|\Delta_j\Lambda^s(\Delta_kf\cdot\nabla\tilde{\Delta}_k)g\|_{L^p}\\
=&K_{21}+K_{22}+K_{23}+K_{24}.
\end{aligned}
\end{equation*}
Since
$$\widehat{\Delta_j \Lambda^s f}(\xi)=\varphi(2^{-j}\xi)|\xi|^s\widehat{f}(\xi)=2^{js}\varphi(2^{-j}\xi)|2^{-j}\xi|^s\widehat{f}(\xi),$$
we can represent $\Delta_j\Lambda^s f$ as a convolution, namely,
$$\Delta_j\Lambda^s f(x)=\{2^{j(d+s)}\zeta(2^j\cdot)\star f \}(x),\ {\rm for \ some}\  \zeta\in \mathcal{S}(\R^d).$$
(\ref{SCE}) can not be used to bound $K_{21}$, but thanks to (\ref{lWu}), and using Bernstein's inequality,
\begin{equation*}
\begin{aligned}
K_{21}\le& C\|x2^{j(d+s)}\zeta(2^jx)\|_{L^1}\sum_{|k-j|\le4}\|\nabla S_{k-1}f\|_{L^\infty}\|\nabla \Delta_k g\|_{L^p}\\
\le& C2^{j(s-1)}\|\nabla f\|_{L^\infty}\sum_{|k-j|\le4}\|\nabla \Delta_k g\|_{L^p}\\
\le& C2^{-j\sigma}\|\nabla f\|_{L^\infty}\sum_{|k-j|\le4}2^{(j-k)(s+\sigma-1)}2^{k(s+\sigma-1)}\|\nabla \Delta_k g\|_{L^p}\\
=& C2^{-j\sigma}\|\nabla f\|_{L^\infty}\| g\|_{\dot{B}_{p,r}^{\sigma+s}}\sum_{|k-j|\le4}2^{(j-k)(s+\sigma-1)}c_k\\
= & Cc_j2^{-j\sigma}\|\nabla f\|_{L^\infty}\| g\|_{\dot{B}_{p,r}^{\sigma+s}}.
\end{aligned}
\end{equation*}
For the terms $K_{22}$ and $K_{23}$, with a similar procedure as the estimate of  $K_{12}$ and $K_{13}$, respectively, we have
\begin{equation*}
\begin{aligned}
K_{22}\le& C2^{js}\sum_{|k-j|\le4}\|\nabla S_{k-1} g\|_{L^\infty}\|\Delta_k f\|_{L^p}\\
\le& C2^{js}\|\nabla  g\|_{L^\infty}\sum_{|k-j|\le4}\|\Delta_k f\|_{L^p}\\
\le& C2^{-j\sigma}\|\nabla  g\|_{L^\infty}\sum_{|k-j|\le4}2^{(j-k)(s+\sigma)}2^{k(s+\sigma)}\|\Delta_k f\|_{L^p}\\
\le& C2^{-j\sigma}\|\nabla  g\|_{L^\infty}\| f\|_{\dot{B}_{p,r}^{\sigma+s}}\sum_{|k-j|\le4}2^{(j-k)(s+\sigma)}c_k\\
=& Cc_j2^{-j\sigma}\|\nabla  g\|_{L^\infty}\| f\|_{\dot{B}_{p,r}^{\sigma+s}},
\end{aligned}
\end{equation*}
and
\begin{equation*}
\begin{aligned}
K_{23}\le& C2^{j(1+s)}\|\Delta_j g\|_{L^p}\sum_{k\ge j-2}\|\Delta_k f\|_{L^\infty}\\
\le& C2^{j(1+s)}\|\Delta_j g\|_{L^p}\sum_{k\ge j-2}2^{-k}\|\nabla\Delta_k f\|_{L^\infty}\\
\le& C2^{js}\|\Delta_j g\|_{L^p}\|\nabla f\|_{L^\infty}\\
\le& C2^{-j\sigma}2^{j(s+\sigma)}\|\Delta_j g\|_{L^p}\|\nabla f\|_{L^\infty}\\
=& Cc_j2^{-j\sigma}\|\nabla f\|_{L^\infty}\| g\|_{\dot{B}_{p,r}^{\sigma+s}}.
\end{aligned}
\end{equation*}
Using ${\rm  div} f=0$, Bernstein's inequality, H\"{o}lder's inequality and Young's inequality for series,
\begin{equation*}
\begin{aligned}
K_{24}\le& C2^{j(s+1)}\sum_{k\ge j-3}\|\Delta_k f\|_{L^\infty}\|\tilde{\Delta}_k g\|_{L^p}\\
\le& C2^{j(s+1)}\sum_{k\ge j-3}2^{-k}\|\nabla\Delta_k f\|_{L^\infty}\|\tilde{\Delta}_k g\|_{L^p}\\
\le& C2^{j(s+1)}\|\nabla f\|_{L^\infty}\sum_{k\ge j-3}2^{-k}\|\tilde{\Delta}_k g\|_{L^p}\\
\le& C2^{-j\sigma}\|\nabla f\|_{L^\infty}\sum_{k\ge j-3}2^{(j-k)(s+1+\sigma)}2^{k(s+\sigma)}\|\tilde{\Delta}_k g\|_{L^p}\\
=& C2^{-j\sigma}\|\nabla f\|_{L^\infty}\| g\|_{\dot{B}_{p,r}^{\sigma+s}}\sum_{k\ge j-3}2^{(j-k)(s+1+\sigma)}c_k\ \ (s+\sigma>-1)\\
=& Cc_j2^{-j\sigma}\|\nabla f\|_{L^\infty}\| g\|_{\dot{B}_{p,r}^{\sigma+s}}.
\end{aligned}
\end{equation*}
It is easy to deduce that
$$K_2\le \sum_{1\le i\le 4}(\sum_{j\in\Z}2^{j\sigma r}K_{2i}^r)^\frac{1}{r}\le C\left\{\|\nabla f\|_{L^\infty}\| g\|_{\dot{B}_{p,r}^{\sigma+s}}+\|\nabla g\|_{L^\infty}\| f\|_{\dot{B}_{p,r}^{\sigma+s}}\right\}.$$
Combining with the estimates of $K_1$ and $K_2$  leads the desired estimate (\ref{CE}).
\end{proof}

\vskip .3in
\section{Proof of Theorem \ref{main}}
\label{sec:mainproof}
In this section, we will prove Theorem \ref{main} by splitting the details into two steps. In step 1, we show the  local well-posedness for (\ref{TCM}) in brief. More precisely, we only give local a priori bound since other details can be proved by standard method, see Chapter 3 \cite{MB}. In step 2, we find the hidden thermal diffusion by  exploiting the structure as we described in section \ref{Introduction} and then obtain the global bound with small data by  using the commutator estimate (\ref{CE}) in section \ref{sec:Preliminaries}.
\vskip .1in
Now, we begin the proof.\\
{\bf Step 1. Local a priori bound.} Thanks to the cancelation property,
$$(\nabla \theta|v)+({\rm div}v|\theta)=0,$$
it is easy to get the $L^2$ bound of $(u,v,\theta)$:
\begin{equation}\label{3.1}
\frac{1}{2}\frac{d}{dt}\|(u,v,\theta)\|_{L^2}^2+\alpha\|(u,v)\|_{L^2}^2+\eta\|\nabla v\|_{L^2}^2=0.
\end{equation}
By the standard energy estimate, and noting
$$(\nabla \theta|v)_{\dot{H}^s}+({\rm div}v|\theta)_{\dot{H}^s}=0,$$
with (\ref{3.1}), we deduce that
\begin{equation}\label{3.2}
\begin{aligned}
\frac{1}{2}\frac{d}{dt}&\|(u,v,\theta)\|_{H^s}^2+\alpha\|(u,v)\|_{H^s}^2+\eta\|\nabla v\|_{H^s}^2\\
=&-(u\cdot\nabla u|u)_{\dot{H}^s}-({\rm div}(v\otimes v)|u)_{\dot{H}^s}-(u\cdot\nabla v|v)_{\dot{H}^s}\\
&-(v\cdot\nabla u|v)-(u\cdot\nabla \theta|\theta)_{\dot{H}^s}\\
:\stackrel{def}{=}&\sum_{i=1}^5I_i.
\end{aligned}
\end{equation}
By integrating by parts, we get
$$(u\cdot\nabla \Lambda^s u|\Lambda^s u)=(u\cdot\nabla\Lambda^s v|\Lambda^s v)=(u\cdot\nabla\Lambda^s \theta|\Lambda^s \theta)=0$$
and
$$(v\cdot\nabla\Lambda^s u|\Lambda^s v)+(v\cdot\nabla\Lambda^s v|\Lambda^s u)=-({\rm div}v|\Lambda^s u\cdot\Lambda^s v),$$
from which, with the equality ${\rm div}(v\otimes v)=v\ {\rm div}v+v\cdot\nabla v$, we obtain
\begin{equation*}
\begin{aligned}
I_1=&([\Lambda^s,u\cdot\nabla]u|\Lambda^s u),\ \ I_3=([\Lambda^s,u\cdot\nabla]v|\Lambda^s v),\ \ I_5=([\Lambda^s,u\cdot\nabla]\theta|\Lambda^s\theta),\\
I_2+I_4=&({\rm div}v|v\cdot u)_{\dot{H}^s}+(v\cdot\nabla v|u)_{\dot{H}^s}+(v\cdot\nabla u|v)_{\dot{H}^s}\\
=& (v\ {\rm div}v| u)_{\dot{H}^s}+([\Lambda^s,v\cdot\nabla]v|\Lambda^s u)+([\Lambda^s,v\cdot\nabla]u|\Lambda^s u)
-({\rm div}v|\Lambda^s u\cdot\Lambda^s v ).\\
\end{aligned}
\end{equation*}
Using H\"{o}lder's inequality, (\ref{kp}), (\ref{kpv}) and Young's inequality follows that
\begin{equation*}
\begin{aligned}
I_1\le& C\|\nabla u\|_{L^\infty}\|\Lambda^s u\|_{L^2}^2,\\
I_2+I_4\le& C(\|{\rm div}v\|_{L^\infty}\|v\|_{\dot{H}^s}+\|v\|_{L^\infty}\|{\rm div}v\|_{\dot{H}^s})\|u\|_{\dot{H}^s}\\
&+C(\|\nabla u\|_{L^\infty}+\|\nabla v\|_{L^\infty})(\|u\|_{\dot{H}^s}^2+\|v\|_{\dot{H}^s}^2)\\
&+C\|{\rm div}v\|_{L^\infty}\|u\|_{\dot{H}^s}\|v\|_{\dot{H}^s}\\
\le& C(\|\nabla u\|_{L^\infty}+\|\nabla v\|_{L^\infty})(\|u\|_{\dot{H}^s}^2+\|v\|_{\dot{H}^s}^2)\\
&+C\eta\|v\|_{L^\infty}^2\|u\|_{\dot{H}^s}^2+\frac{\eta}{2}\|\nabla v\|_{\dot{H}^s}^2,\\
I_3\le& C(\|\nabla u\|_{L^\infty}+\|\nabla v\|_{L^\infty})(\|u\|_{\dot{H}^s}^2+\|v\|_{\dot{H}^s}^2),\\
I_5\le& C(\|\nabla u\|_{L^\infty}+\|\nabla \theta\|_{L^\infty})(\|u\|_{\dot{H}^s}^2+\|\theta\|_{\dot{H}^s}^2).
\end{aligned}
\end{equation*}
Combining with the above estimates in (\ref{3.2}), thanks to
$$\|f\|_{\dot{H}^s}\le C\|\nabla f\|_{H^{s-1}},\ \ \|\nabla f\|_{L^\infty}\le C\|\nabla f\|_{H^{s-1}},\ \ s>2,$$
we have
\begin{equation}\label{original}
\begin{aligned}
\frac{1}{2}\frac{d}{dt}&\|(u,v,\theta)\|_{H^s}^2+\alpha\|(u,v)\|_{H^s}^2+\frac{\eta}{2}\|\nabla v\|_{H^s}^2\\
&\le C(\|\nabla u\|_{H^{s-1}}^2+\|\nabla v\|_{H^{s-1}}^2+\|\nabla \theta\|_{H^{s-1}}^2)^\frac{3}{2}+C\eta\|v\|_{H^s}^2\|u\|_{H^s}^2\\
&\le C(\eta)\left(\|(u,v,\theta)\|_{H^s}^3+\|(u,v)\|_{H^s}^4\right),
\end{aligned}
\end{equation}
which implies that there exists a $T_0=T_0(\eta,\|(u_0,v_0,\theta_0)\|_{H^s})>0$ such that $\forall\ t\in(0,T_0],$
$$\|u(t)\|_{H^s}^2+\|v(t)\|_{H^s}^2+\|\theta(t)\|_{H^s}^2\le C(\eta,\ T_0,\|(u_0,v_0,\theta_0)\|_{H^s})$$
and then
$$\alpha\int_{0}^{t}\|(u,v)(\tau)\|_{H^s}^2d\tau+\frac{\eta}{2}\int_{0}^{t}\|\nabla v(\tau)\|_{H^s}^2d\tau\le C(\eta,\ T_0,\|(u_0,v_0,\theta_0)\|_{H^s}).$$
So we can get the local a priori bound and then obtain the local well-posedness of (\ref{TCM}) by standard method.

{\bf Step 2. Global well-posedness.} Thanks to step 1, it suffices to give the global a priori bound. Denote
$$\mathcal{R}:\stackrel{def}{=}\Lambda^{-1}{\rm div},\ \ \Omega: \stackrel{def}{=}\mathcal{R}v+\eta\Lambda \theta,$$
and we will use the $L^2$ bound of Riesz transform $\mathcal{R}$ in some places of this step.\\
Applying the operator $\mathcal{R}$ and $\Lambda$ to the second equation and third equation of (\ref{TCM}), respectively, we get
\begin{equation}\label{3.3}
\partial_{t}\mathcal{R}v+u\cdot\nabla \mathcal{R}v+\alpha\mathcal{R}v+\eta \Lambda {\rm div}v+\Lambda \theta=-[\mathcal{R},u\cdot\nabla]v-\mathcal{R}(v\cdot\nabla u)
\end{equation}
\begin{equation}\label{3.4}
\partial_t \Lambda \theta+u\cdot\nabla \Lambda\theta-\Lambda{\rm div}v=-[\Lambda,u\cdot\nabla]\theta.
\end{equation}
Multiplying (\ref{3.4}) by $\eta$ and adding the resulting equation to (\ref{3.3}) lead
\begin{equation}\label{3.5}
\partial_t \Omega+u\cdot\nabla \Omega+\frac{1}{\eta}\Omega=(\frac{1}{\eta}-\alpha)\mathcal{R}v-[\mathcal{R},u\cdot\nabla]v-\mathcal{R}(v\cdot\nabla u)-\eta[\Lambda,u\cdot\nabla]\theta.
\end{equation}
Multiplying (\ref{3.5}) by $\Omega$, integrating in $\R^2$ and using
$$(u\cdot\nabla \Omega|\Omega)=0$$
 follows
\begin{equation}\label{3.6}
\begin{aligned}
\frac{1}{2}\frac{d}{dt}\|\Omega\|_{L^2}^2&+\frac{1}{\eta}\|\Omega\|_{L^2}^2=(\frac{1}{\eta}-\alpha)(\mathcal{R}v|\Omega)\\
&-([\mathcal{R},u\cdot\nabla]v|\Omega)-(\mathcal{R}(v\cdot\nabla u)|\Omega)-\eta([\Lambda,u\cdot\nabla]\theta|\Omega)\\
:\stackrel{def}{=}&J_1+J_2+J_3+J_4.
\end{aligned}
\end{equation}
By H\"{o}lder's inequality and  Young's inequality,
\begin{equation*}
\begin{aligned}
|J_1|\le& |\frac{1}{\eta}-\alpha|\|v\|_{L^2}\|\Omega\|_{L^2}\le C\eta|\frac{1}{\eta}-\alpha|^2\|v\|_{L^2}^2+\frac{1}{8\eta}\|\Omega\|_{L^2}^2,\\
|J_2|\le& \|[\mathcal{R},u\cdot\nabla]v\|_{L^2}\|\Omega\|_{L^2}\le C\|u\|_{L^\infty}\|\nabla v\|_{L^2}\|\Omega\|_{L^2}\\
\le& C\eta\|u\|_{L^\infty}^2\|\nabla v\|_{L^2}^2+\frac{1}{8\eta}\|\Omega\|_{L^2}^2,\\
|J_3|\le& C\|\mathcal{R}(v\cdot\nabla u)\|_{L^2}\|\Omega\|_{L^2}\le C\eta\|v\|_{L^\infty}^2\|\nabla u\|_{L^2}^2+\frac{1}{8\eta}\|\Omega\|_{L^2}^2.
\end{aligned}
\end{equation*}
By (\ref{kpv}), using $\eta\Lambda \theta=\Omega-\mathcal{R}v$ and Young's inequality,
\begin{equation*}
\begin{aligned}
|J_4|\le& C\eta(\|\nabla u\|_{L^\infty}\|\Lambda \theta\|_{L^2}+\|\nabla \theta\|_{L^\infty}\|\Lambda u\|_{L^2})\|\Omega\|_{L^2}\\
\le& C\|\nabla u\|_{L^\infty}(\|\Omega\|_{L^2}+\|\mathcal{R}v\|_{L^2})\|\Omega\|_{L^2}
+C\eta\|\nabla \theta\|_{L^\infty}\|\Lambda u\|_{L^2}\|\Omega\|_{L^2}\\
\le& C\eta\|\nabla u\|_{L^\infty}^2(\|v\|_{L^2}^2+\|\Omega\|_{L^2}^2)+C\eta^3
\|\nabla \theta\|_{L^\infty}^2\|\Lambda u\|_{L^2}^2+\frac{1}{8\eta}\|\Omega\|_{L^2}^2.
\end{aligned}
\end{equation*}
Inserting the above estimates into (\ref{3.6}) and absorbing the four  $\frac{1}{8\eta}\|\Omega\|_{L^2}^2$ by the left hand side of the resulting
inequality, we have
\begin{equation}\label{3.7}
\begin{aligned}
\frac{1}{2}\frac{d}{dt}\|\Omega\|_{L^2}^2&+\frac{1}{2\eta}\|\Omega\|_{L^2}^2\le C\eta|\frac{1}{\eta}-\alpha|^2\|v\|_{L^2}^2\\
&+C(\eta+\eta^3)(\|(u,\nabla u,v)\|_{L^\infty}^2+\|\nabla \theta\|_{L^\infty}^2)(\|(u,v)\|_{H^1}^2+\|\Omega\|_{L^2}^2)\\
&\le C\eta|\frac{1}{\eta}-\alpha|^2\|v\|_{L^2}^2+C(\eta+\eta^3)\|(u,v,\theta)\|_{H^s}^2(\|(u,v)\|_{H^1}^2+\|\Omega\|_{L^2}^2).
\end{aligned}
\end{equation}
Next, we give the $\dot{H}^{s-1}$ bound of $\Omega$. Applying $\Lambda^{s-1}$ to the both sides of (\ref{3.5}), and taking the inner product with $\Lambda^{s-1}\Omega$, while thanks to
$$(u\cdot\nabla\Lambda^{s-1}\Omega|\Lambda^{s-1}\Omega)=0,$$
we have
\begin{equation}\label{3.8}
\begin{aligned}
\frac{1}{2}\frac{d}{dt}\|\Omega\|_{\dot{H}^{s-1}}^2&+\frac{1}{\eta}\|\Omega\|_{\dot{H}^{s-1}}^2=(\frac{1}{\eta}-\alpha)(\mathcal{R}v|\Omega)_{\dot{H}^{s-1}}
-([\Lambda^{s-1},u\cdot\nabla]\Omega|\Lambda^{s-1}\Omega)\\
&-([\mathcal{R},u\cdot\nabla]v|\Omega)_{\dot{H}^{s-1}}-(\mathcal{R}(v\cdot\nabla u)|\Omega)_{\dot{H}^{s-1}}
-\eta([\Lambda,u\cdot\nabla]\theta|\Omega)_{\dot{H}^{s-1}}\\
:\stackrel{def}{=}&L_1+L_2+L_3+L_4+L_5.
\end{aligned}
\end{equation}
By H\"{o}lder's inequality and Young's inequality,
$$|L_1|\le C|\frac{1}{\eta}-\alpha|\|\Lambda^{s-1} v\|_{L^2}\|\Omega\|_{\dot{H}^{s-1}}\le C\eta|\frac{1}{\eta}-\alpha|^2\|v\|_{\dot{H}^{s-1}}^2+\frac{1}{8\eta}\|\Omega\|_{\dot{H}^{s-1}}^2.$$
By H\"{o}lder's inequality, (\ref{kpv}),  Young's inequality and using $\eta\Lambda \theta=\Omega-\mathcal{R}v$ ,
\begin{equation*}
\begin{aligned}
|L_2|\le& \|[\Lambda^{s-1},u\cdot\nabla]\theta\|_{L^2}\|\Lambda^{s-1} \Omega\|_{L^2}\\
\le& C(\|\nabla u\|_{L^\infty}\|\Lambda^{s-1}\theta\|_{L^2}+\|\nabla \theta\|_{L^\infty}\|\Lambda^{s-1}u\|_{L^2})\|\Lambda^{s-1}\Omega\|_{L^2}\\
\le& C\eta (\|\nabla u\|_{L^\infty}^2\|\Lambda^{s-1}\theta\|_{L^2}^2+\|\nabla \theta\|_{L^\infty}^2\|\Lambda^{s-1}u\|_{L^2}^2)+\frac{1}{8\eta}\|\Omega\|_{\dot{H}^{s-1}}^2\\
\le& C\eta (\|\nabla u\|_{L^\infty}^2\|\Lambda \theta\|_{\dot{H}^{s-2}}^2+\|\nabla \theta\|_{L^\infty}^2\|\Lambda^{s-1}u\|_{L^2}^2)+\frac{1}{8\eta}\|\Omega\|_{\dot{H}^{s-1}}^2\\
\le& \frac{C}{\eta}\|\nabla u\|_{L^\infty}^2(\|\Omega\|_{\dot{H}^{s-2}}^2+\|v\|_{\dot{H}^{s-2}}^2)
+C\eta\|\nabla \theta\|_{L^\infty}^2\|\Lambda^{s-1}u\|_{L^2}^2+\frac{1}{8\eta}\|\Omega\|_{\dot{H}^{s-1}}^2.
\end{aligned}
\end{equation*}
By  H\"{o}lder's inequality, (\ref{kp}) and Young's inequality, $\forall\ \iota\in (0,s-2)$,
\begin{equation*}
\begin{aligned}
|L_3|\le& \|\mathcal{R}(u\cdot\nabla v)-u\cdot\nabla\mathcal{R}v\|_{\dot{H}^{s-1}}\|\Omega\|_{\dot{H}^{s-1}}\\
\le& C\left\{\|u\|_{L^\infty}\|v\|_{\dot{H}^s}+(\|\nabla v\|_{L^\infty}+\|\nabla \mathcal{R}v\|_{L^\infty})\|u\|_{\dot{H}^{s-1}}\right\}\|\Omega\|_{\dot{H}^{s-1}}\\
\le& C(\|u\|_{L^\infty}\|v\|_{\dot{H}^s}+\|\nabla v\|_{H^{1+\iota}}\|u\|_{\dot{H}^{s-1}})\|\Omega\|_{\dot{H}^{s-1}}\ (\iota\in (0,s-2))\\
\le& C\eta(\|u\|_{L^\infty}^2\|v\|_{\dot{H}^s}^2+\|\nabla v\|_{H^{1+\iota}}^2\|u\|_{\dot{H}^{s-1}}^2)+\frac{1}{8\eta}\|\Omega\|_{\dot{H}^{s-1}}^2.
\end{aligned}
\end{equation*}
Similarly,
\begin{equation*}
\begin{aligned}
|L_4|\le& C\|v\cdot\nabla u\|_{\dot{H}^{s-1}}\|\Omega\|_{\dot{H}^{s-1}}\\
\le& C(\|v\|_{L^\infty}\|\nabla u\|_{\dot{H}^{s-1}}+\|\nabla u\|_{L^\infty}\|v\|_{\dot{H}^{s-1}})\|\Omega\|_{\dot{H}^{s-1}}\\
\le& C\eta(\|v\|_{L^\infty}^2\| u\|_{\dot{H}^{s}}^2+\|\nabla u\|_{L^\infty}\|v\|_{\dot{H}^{s-1}}^2)+\frac{1}{8\eta}\|\Omega\|_{\dot{H}^{s-1}}^2.
\end{aligned}
\end{equation*}
By  H\"{o}lder's inequality and  thanks to the commutator estimates (\ref{CE}) with $d=2$, $s=1$, $\sigma=s-1$, $p=r=2$,
\begin{equation*}
\begin{aligned}
|L_5|\le& \eta(\|\nabla \theta\|_{L^\infty}\|u\|_{\dot{H}^s}+\|\nabla u\|_{L^\infty}\|\theta\|_{\dot{H}^s})\|\Omega\|_{\dot{H}^s}\\
\le& C\eta^3(\|\nabla \theta\|_{L^\infty}^2\|u\|_{\dot{H}^s}^2+\|\nabla u\|_{L^\infty}^2\|\theta\|_{\dot{H}^s}^2)+\frac{1}{8\eta}\|\Omega\|_{\dot{H}^s}^2\\
\le& C\eta^3\left\{\|\nabla \theta\|_{L^\infty}^2\|u\|_{\dot{H}^s}^2+\frac{1}{\eta^2}\|\nabla u\|_{L^\infty}^2(\|\Omega\|_{\dot{H}^{s-1}}^2+\|v\|_{\dot{H}^{s-1}}^2)\right\}+\frac{1}{8\eta}\|\Omega\|_{\dot{H}^s}^2\\
\le& C\eta^3\|\nabla \theta\|_{L^\infty}^2\|u\|_{\dot{H}^s}^2+C\eta\|\nabla u\|_{L^\infty}^2(\|\Omega\|_{\dot{H}^{s-1}}^2+\|v\|_{\dot{H}^{s-1}}^2)
+\frac{1}{8\eta}\|\Omega\|_{\dot{H}^s}^2.
\end{aligned}
\end{equation*}
Combining the estimates of $L_l$ $(l=1,2,3,4,5)$ in (\ref{3.8}), it is easy to deduce that
\begin{equation*}
\begin{aligned}
\frac{1}{2}\frac{d}{dt}&\|\Omega\|_{\dot{H}^{s-1}}^2+\frac{1}{2\eta}\|\Omega\|_{\dot{H}^{s-1}}^2\le C\eta|\frac{1}{\eta}-\alpha|^2\|v\|_{\dot{H}^{s-1}}^2
+C(\frac{1}{\eta}+\eta+\eta^3)\\
\times&\left(\|(u,v,\nabla u,\nabla v)\|_{L^\infty}^2+\|\nabla \theta\|_{L^\infty}^2+\|\nabla v\|_{H^{1+\iota}}^2\right)
\left(\|(u,v)\|_{H^s}^2+\|\Omega\|_{\dot{H}^{s-1}}^2\right),
\end{aligned}
\end{equation*}
and thanks to (\ref{3.7}), we obtain
\begin{equation}\label{3.10}
\begin{aligned}
\frac{1}{2}\frac{d}{dt}\|\Omega\|_{H^{s-1}}^2&+\frac{1}{2\eta}\|\Omega\|_{H^{s-1}}^2\le C\eta|\frac{1}{\eta}-\alpha|^2\|v\|_{H^{s-1}}^2
+C(\frac{1}{\eta}+\eta+\eta^3)\\
\times&\|(u,v,\theta)\|_{H^s}^2
\left(\|(u,v)\|_{H^s}^2+\|\Omega\|_{H^{s-1}}^2\right).
\end{aligned}
\end{equation}
Using $\eta\Lambda \theta=\Omega-\mathcal{R}v$ again, we can rewrite  (\ref{original}) as follows:
\begin{equation}\label{3.11}
\begin{aligned}
\frac{1}{2}\frac{d}{dt}&\|(u,v,\theta)\|_{H^s}^2+\alpha\|(u,v)\|_{H^s}^2+\frac{\eta}{2}\|\nabla v\|_{H^s}^2
\le C\|(\nabla u,\nabla v,\nabla \theta)\|_{H^{s-1}}\\
&\ \ \times\left(\|\nabla u\|_{H^{s-1}}^2+\|\nabla v\|_{H^{s-1}}^2+\|\Lambda \theta\|_{H^{s-1}}^2\right)+C\eta\|v\|_{H^s}^2\|u\|_{H^s}^2\\
&\le C(1+\frac{1}{\eta^2})\|(u,v,\theta)\|_{H^s}\|(\nabla u,\nabla v,v,\Omega)\|_{H^{s-1}}^2+C\eta\|v\|_{H^s}^2\|u\|_{H^s}^2.\\
\end{aligned}
\end{equation}
Multiplying (\ref{3.11}) by
$$M:\stackrel{def}{=}\frac{2C\eta|\frac{1}{\eta}-\alpha|^2}{\alpha},$$
adding the resulting inequality to (\ref{3.10}), and then absorbing the term $C\eta|\frac{1}{\eta}-\alpha|^2\|v\|_{H^{s-1}}^2$,
we obtain
\begin{equation*}
\begin{aligned}
\frac{1}{2}\frac{d}{dt}&\left\{M\|(u,v,\theta)\|_{H^s}^2+\|\Omega\|_{H^{s-1}}^2\right\}+\frac{M\alpha}{2}
\|(u,v)\|_{H^s}^2+\frac{M\eta}{2}\|\nabla v\|_{H^s}^2+\frac{1}{2\eta}\|\Omega\|_{H^{s-1}}^2\\
\le& C(M,\eta)\left(\|(u,v,\theta)\|_{H^s}^2+\|(u,v,\theta)\|_{H^s}\right)\left(\|(u,v)\|_{H^s}^2+\|\Omega\|_{H^{s-1}}^2\right).
\end{aligned}
\end{equation*}
Denote
$$m:\stackrel{def}{=}\min\left\{\frac{M\alpha}{2},\ \frac{\eta M}{2},\ \frac{1}{2\eta}\right\},$$
by Young's inequality, it is easy to get
\begin{equation}\label{3.12}
\begin{aligned}
\frac{d}{dt}&\left\{M\|(u,v,\theta)\|_{H^s}^2+\|\Omega\|_{H^{s-1}}^2\right\}+2m(
\|(u,\nabla v,v)\|_{H^s}^2+\|\Omega\|_{H^{s-1}}^2)\\
\le& C(M,\eta)\left(\|(u,v,\theta)\|_{H^s}^2+\|(u,v,\theta)\|_{H^s}\right)\left(\|(u,v)\|_{H^s}^2+\|\Omega\|_{H^{s-1}}^2\right)\\
\le& C(M,m,\eta)\|(u,v,\theta)\|_{H^s}^2\left(\|(u,v)\|_{H^s}^2+\|\Omega\|_{H^{s-1}}^2\right)+m(
\|(u,v)\|_{H^s}^2+\|\Omega\|_{H^{s-1}}^2).
\end{aligned}
\end{equation}
Absorbing the second term on the  right hand side in the last inequality of (\ref{3.12}), then integrating in time yields  $\forall\ t>0,$\\
\begin{equation}\label{3.13}
\begin{aligned}
M\|(u,v,\theta)(t)\|_{H^s}^2&+\|\Omega(t)\|_{H^{s-1}}^2+m\int_{0}^{t}\|(u,\nabla v,v)(\tau)\|_{H^s}^2+\|\Omega(\tau)\|_{H^{s-1}}^2 d\tau\\
\le& C(M,m,\eta)\int_{0}^{t}\|(u,v,\theta)(\tau)\|_{H^s}^2\left(\|(u,v)(\tau)\|_{H^s}^2+\|\Omega(\tau)\|_{H^{s-1}}^2\right)d\tau\\
&+M\|(u_0,v_0,\theta_0)\|_{H^s}^2+\|\Omega_0\|_{H^{s-1}}^2,
\end{aligned}
\end{equation}
With small data (\ref{small}), choosing $\epsilon$ to be so small that
$$M\|(u_0,v_0,\theta_0)\|_{H^{s}}^2+\|\Omega_0\|_{H^{s-1}}^2\le \frac{Mm}{2C(M,m,\eta)}$$
which implies that $\|(u_0,v_0,\theta_0)\|_{H^s}^2< \frac{m}{2C(M,m,\eta)}$. Suppose there exists a first time $T^\star$ such that $\forall t\in(0,T^\star),$
$$M\|(u,v,\theta)(t)\|_{H^{s}}^2+\|\Omega(t)\|_{H^{s-1}}^2\le \frac{Mm}{2C(M,m,\eta)}$$
and
\begin{equation}\label{3.14}
\lim_{\epsilon_1\searrow0}M\|(u,v,\theta)(T^\star-\epsilon_1)\|_{H^{s}}^2+\|\Omega(T^\star-\epsilon_1)\|_{H^{s-1}}^2> \frac{Mm}{2C(M,m,\eta)}.
\end{equation}
However, from (\ref{3.13}), we can deduce
\begin{equation}\label{3.15}
\begin{aligned}
&\ \ M\|(u,v,\theta)(T^\star-\epsilon_1)\|_{H^s}^2+\|\Omega(T^\star-\epsilon_1)\|_{H^{s-1}}^2\\
&+\frac{m}{2}\int_{0}^{T^\star-\epsilon_1}\|(u,\nabla v,v)(\tau)\|_{H^s}^2+\|\Omega(\tau)\|_{H^{s-1}}^2 d\tau\\
\le& M\|(u_0,v_0,\theta_0)\|_{H^s}^2+\|\Omega_0\|_{H^{s-1}}^2<\frac{Mm}{2C(M,m,\eta)},
\end{aligned}
\end{equation}
which yields that
$$M\|(u,v,\theta)(T^\star-\epsilon_1)\|_{H^s}^2+\|\Omega(T^\star-\epsilon_1)\|_{H^{s-1}}^2<\frac{Mm}{2C(M,m,\eta)},$$
from which, and taking $\epsilon_1\searrow 0$, we get a contradiction with (\ref{3.14}). Therefore,
$$\lim_{\epsilon_1\searrow 0}M\|(u,v,\theta)(T^\star-\epsilon_1)\|_{H^s}^2+\|\Omega(T^\star-\epsilon_1)\|_{H^{s-1}}^2\le \frac{Mm}{2C(M,m,\eta)},
$$
which indicates under condition (\ref{small}), we have a global solution $(u,v,\theta)$ satisfying $\forall\ t>0$,
$$\|(u,v,\theta)(t)\|_{H^s}\le \frac{m}{2C(M,m,\eta)}=C(\alpha,\eta)<\infty,$$
and then using (\ref{3.13}), we can also obtain
$$\int_{0}^{t}\|(u,\nabla v,v)(\tau)\|_{H^s}^2 d\tau \le \frac{M}{2C(M,m,\eta)}=C(\alpha,\eta)<\infty.$$
This completes the proof of Theorem \ref{main}.



\vskip .4in
\begin{thebibliography}{99}

\bibitem{BCD} H. Bahouri, J.-Y. Chemin, R. Danchin, \textit{Fourier Analysis and Nonlinear Partial Differential Equations}, Springer, 2011.

\bibitem{Wu1} C. Cao, J. Wu, Global regularity for the 2D MHD equations with mixed partial dissipation and magnetic diffusion,   {\it  Adv. Math. \bf 226} (2011) 1803-1822.
\bibitem{Wu2} C. Cao, J. Wu, B. Yuan, The 2D incompressible magnetohydrodynamics equations with only magnetic diffusion,   {\it  SIAM J. Math. Anal.  \bf 46} (2014), 588-602.
\bibitem{Chemin} J.-Y. Chemin, D.S. McCormick, J.C. Robinson,  J.L. Rodrigo, Local existence for the non-resistive MHD equations in Besov space,  arXiv:1503.01651v1 [math.AP] 5 Mar 2015.

\bibitem{F5} J. Fan, H. Malaikah, S. Monaquel, G. Nakamura, Y. Zhou, Global cauchy problem of 2D generalized MHD equations, {\it Monatsh. Math. \bf 175} (2014), 127-131

\bibitem{Fefferman} C.L. Fefferman,   D.S.  McCormick, J.C.  Robinson,  J.L. Rodrigo, Higher order commutator estimates and local existence for the non-resistive
MHD equations and related models, {\it J. Funct. Anal. \bf 267} (2014), 1035-1056.


\bibitem{FMP} D.M.W. Feireisl, A.J. Majda, O.M. Pauluis, Large scale dynamics of precipitation fronts in the tropical atmosphere:
 a novel relaxation limit,
{\it  Commun. Math. Sci. \bf 2} (2004), 591-626.


\bibitem{Gill}  A.E. Gill, Some simple solutions for the heat-induced tropical circulation,
{\it  Quart. J. Roy. Meteor. Soc. \bf 106} (1980), 447-462.

\bibitem{Grafakos} L. Grafakos, \textit{Modern Fourier Analysis.} 2nd Edition., Grad. Text in Math., \textbf{250}, Springer-Verlag, 2008.

\bibitem{Hmidi1} T. Humidi, S. Keraani, F. Rousset, Global well-posedness for a  Boussinesq-Navier-Stokes  system with  critical dissipation,   {\it   J. Differential Equations  \bf 249} (2010), 2147-2174.


\bibitem{Hmidi2} T. Humidi, S. Keraani, F. Rousset, Global well-posedness for a  Euler-Boussinesq  system with  critical dissipation,   {\it   Comm. Partial  Differential Equations  \bf 36} (2011), 420-445.


\bibitem{JMWZ} Q. Jiu, C. Miao, J. Wu, Z. Zhang,  The 2D incompressible Boussinesq equations with general critical dissipation,   {\it  SIAM J. Math. Anal.  \bf 46} (2014), 3426-3554.

\bibitem{KP} T. Kato, G. Ponce,  Commutator estimates and the Euler and Navier-
Stokes equations, {\it   Comm. Pure Appl. Math. \bf 41}, (1988), 891-907.

\bibitem{KPV} C. Kenig, G. Ponce, L. Vega, Well-posedness of the initial value problem for the Korteweg-de Vries equation,
{\it  J. Amer. Math. Soc. \bf 4} (1991), 323-347.


\bibitem{LT} J. Li, E.S. Titi, Global well-posedness of strong solutions to a  tropical climate model, arxiv: 1504.05285v1 [math. AP] 21 Apr 2015.


\bibitem{MB} A.J. Majda and A.L. Bertozzi, \textit{Vorticity and Incompressible Flow}, Cambridge University Press, Cambridge, UK, 2001.

\bibitem{Matsuno} T. Matsuno, Quasi-geostrophic motions in the equational area,
{\it  J. Meteor. Soc. Japan \bf 44} (1966), 25-42.

\bibitem{Stein} E.M. Stein, \textit{Singular Integrals and Differentiability Properties of Functions}, Princeton University Press, Princeton, 1970.

\bibitem{Wan} R. Wan, On the uniqueness for the 2D MHD equations without magnetic diffusion,  arXiv:1503.03589v1 [math.AP] 12 Mar 2015.


\bibitem{WuJMFM} J. Wu,  Global regularity for a class of generalized Magnetohydrodynamic equations,   {\it  J. Math. Fluid Mech.  \bf  13} (2011), 295-305.



\end{thebibliography}
\end{document}